\documentclass[12pt,a4paper]{article}
\usepackage{latexsym,amsmath,amssymb}
\date{}
\numberwithin{equation}{section}

\DeclareMathOperator{\co}{co}

\DeclareMathOperator{\card}{card}

\DeclareMathOperator{\inter}{int}

\newtheorem{theo}{Theorem}
\newtheorem{lem}{Lemma}
\newtheorem{cor}{Corollary}

\newenvironment{proof}{{\bf Proof.}}{\hfill$\blacksquare$\\}

\begin{document}

\title{\bf Family of closed convex sets covering faces of the simplex}
\author{\large Horst Kramer, A. B. N\'emeth}
\maketitle
\begin{abstract} Let $S=\co \{a_1,a_2,...,a_{n+1}\}$ be the simplex spanned
by the independent points $a_i\in \mathbb R^n,\,i =1,2,...,n+1,$  and
denote by $S^i$ its face $\co \{a_1,...,a_{i-1},a_{i+1},...,a_{n+1}\}.$
If a family of $n+1$ closed convex sets\\
$A^1,\,A^2,...,A^{n+1}$ in $S$ has the property $S^i\subset
A^i$ for each $i$, then there exists a point $v$ in $S$ which is at the
distance $\varepsilon_0\geq 0$ to every set $A^i.$ If $\varepsilon_0>0,$ then the point $v$
with this property is unique.
\end{abstract}
\section{Introduction}

There is classical literature of the combinatorial and algebraic topology
considering the problem of the covering of a simplex. As the first and most
famous result we mention Sperner's Lemma \cite{SP}. A dual result to the
Lemma of Sperner is the Knaster-Kuratowski-Mazurkiewicz Theorem \cite{KKM}.
These results give raise to extensive investigations and applications, among
which we mention the results in \cite{FA} and \cite{SH}. Some
covering problems with closed convex sets benefit substantially from the
above mentioned issues, as is reflected in \cite{BE}, \cite{GH} and
\cite{FA}. As far as we know there was no attempt in a more direct, more
geometric approach in the convex case, although this special context allows
to obtain specific results. In our recent paper \cite{KN} on families of
convex sets we also follow the line of using essentially the classical
covering theorems with closed sets, but as we shall show, a method developed
there allows to prove an extended convex variant of Sperner's lemma without
combinatorial reasonings. In this note we are doing this. Thus all the
results concerning convex sets in \cite{BE}, \cite{GH}, \cite{FA} and
\cite{KN} can be obtained by using our Theorem 1 in place of Sperner's lemma.

Our note continues the line of our erlier investigations in
\cite{KN1}, \cite{KN2}, \cite{KN3} and \cite{KH} about the existence of equally spaced points
to some families of compact and convex sets in the Euclidean and Minkowski
spaces.

\section{Main results}

Let $S=\co \{a_1,a_2,...,a_{n+1}\}$ be the simplex spanned
by the independent points $a_i\in \mathbb R^n,\,i\in N =\{1,2,...,n+1\},$  and
denote by $S^i$ its face $\co \{a_1,...,a_{i-1},\\a_{i+1},...,a_{n+1}\}.$
Our main result is as follows:

\begin{theo}\label{main}
If a family of $n+1$ closed convex sets
$A^1,\,A^2,...,A^{n+1}$ in $S$ has the property $S^i\subset
A^i$ for each $i$, then there exists a point $v$ in $S$ and an $\varepsilon_0\geq 0$ such that $v$ is at the
distance $\varepsilon_0$ to every set $A^i.$ If $\varepsilon_0>0,$ then the point $v$
with this property is unique.
\end{theo}

A family $\{B^\alpha:\;\alpha\in I\}$ is said to be a {\it face covering for $S$} if
each $B^\alpha$ contains some face $S^i$.

The unicity of the equally spaced point in the first part of Theorem \ref{main}
concludes

\begin{cor}\label{dist}

If $\{B^\alpha:\;\alpha\in I,\}\;\card I\geq n+2$ is a face covering family 
of closed convex sets in $S$ such that every $S^i,\; i=1,2,...,n+1$ 
is contained in some of its sets and each subfamily
with $n+2$ members possesses an equally spaced point in $S$ of positive
distance, then the whole family possesses such a point.

\end{cor}

\begin{cor}\label{hell}
The  face covering family $\mathcal B= \{B^\alpha:\;\alpha\in I\}$
of closed convex sets in $S$ has a nonempty intersection if and only if every subfamily
$B^{\alpha_1},B^{\alpha_2},...,B^{\alpha_{n+1}}$  of it, which
covers the boundary of $S$, covers also $S$.
\end{cor}

\section{The proofs}

Let us consider $N=\{1,2,...,n+1\}$ and a family $\mathcal H
=\{A^1,A^2,...,A^{n+1}\} $ of closed convex sets in $S$. Then
$\mathcal H$ will be called an $\mathcal H$-{\it family}, if  $S^i\subset A^i,\;i\in N$
and $S\setminus \cup _{i\in N} A^i \not= \emptyset.$

We prove Theorem \ref{main} by using two auxiliary lemmas. The first of them is

\begin{lem}\label{kn}
Consider the $\mathcal H$-family $\{A^1,A^2,...,A^{n+1}\}$.
Let $A^i_\varepsilon$ be the $\varepsilon >0$-hull of the set $A^i$ in $S$, i. e., the set of
points in $S$ with the distance $\leq \varepsilon$ from the set $A^i$. Then

\begin{enumerate}
\item There exists an $\varepsilon_0>0$ such that:

(i) $\{A^i_\varepsilon :\, i\in N\}$ is a $\mathcal H$-family for $\varepsilon <\varepsilon_0$,

(ii) $B_\varepsilon =\cap_{i\in N}A^i_\varepsilon \not= \emptyset$ for $\varepsilon
\geq \varepsilon_0$.

\item $B_{\varepsilon _0}$ reduces to a single point $v$, and $\varepsilon_0$ is the
common distance of $v$ to $A^i$.
\item $v$ is the single point which is equally spaced to the sets $A^i$.
\item There holds the relation
$$\varepsilon_0=\sup_{u\in S} \inf_{i\in N} d(u,A^i),$$
where $d(a,A)$ denotes the distance of the point $a$ from the set $A$.
\item If $\{D^1,D^2,...,D^{n+1}\}$ is another $\mathcal H$-family
with $A^i\subset D^i,\;i=1,2,...,n+1$ and $\varepsilon_1$ is the common
distance of the equally spaced point from $D^i$, then $\varepsilon_1\leq \varepsilon_0.$
\end{enumerate}
\end{lem}
\begin{proof}

Our proof consists in fact from gathering some considerations in the proofs in
\cite{KN}.

Let $A^i_\varepsilon$ be the closed $\varepsilon \geq 0$-hull of the set $A^i$ in $S$, i. e., the
set of points in $S$ with the distance $\leq \varepsilon$ from the set $A^i$.
Denote $B_\varepsilon =\cap_{i\in N}A^i_\varepsilon$. Then
$$\Omega = \{\varepsilon :\;B_\varepsilon \not= \emptyset\}$$ is obviously not empty
and $\varepsilon_0=\inf \Omega$ is well defined. Indeed, since $B_\varepsilon,\; \varepsilon \in \Omega$
are nonempty compact convex sets, we have that
$$B_{\varepsilon_0}=\cap_{\varepsilon \in \Omega} B_\varepsilon $$
is a nonempty compact convex set.

Let us suppose $\varepsilon_0=0$. Then $B_0=\cap_{i\in N} A^i \not= \emptyset$ and for
$v\in B_0$ we would have $\co \{S^i,v\}\subset A^i,\;i\in N$
the later implying that $\cup_{i\in N}\co \{S^i,v\}$ covers $S$ and therefore
also $\cup_{i\in N}A^i$ covers S, which is a contradiction.  Thus $\varepsilon_0>0.$

We shall show first that no point of $B_{\varepsilon_0}$ can be an interior point of
some $A^i_{\varepsilon_0}$. Assuming the contrary, e.g. that $b\in B_{\varepsilon_0}\cap
\inter A^i_{\varepsilon_0}$ we have first of all that $d(b,A^i)<\varepsilon_0$ and
$d(b,A^j)\leq \varepsilon_0, \;j\in N.$ Since $\cap_{j\in N\setminus \{i\}}A^j$ is nonempty,
$\varepsilon_0>0$, the set $\cap_{j\in N\setminus \{i\}}A^j_{\varepsilon_0}$ is convex
and has a nonempty interior. Now, $b\in \cap_{j\in N\setminus \{i\}}A^j_{\varepsilon_0}$ and each of
its neighborhoods contains interior points of $\cap_{j\in N\setminus \{i\}}A^j_{\varepsilon_0}$.
Hence so does $\inter A^i_{\varepsilon_0}$. Let be $x$ a such point. Then
$d(x,A^j)<\varepsilon_0, j\in N$. Denote by $\delta = \sup \{d(x,A^j):\,j\in N\}.$
It follows that $x\in B_\delta$ with $\delta <\varepsilon_0$, in contradiction with
the definition of $\varepsilon_0$.

Thus $B_{\varepsilon_0}$ is on the boundary of every $A^i_{\varepsilon_0}$. Hence:
\begin{center}$d(b,A^j)=\varepsilon_0,\;\forall \,j\in N\;\forall \, b\in B_{\varepsilon_0}$.\end{center}

If $B_{\varepsilon_0}$ would contain two distinct points, $b_1$ and $b_2$, the
line segment determined by these two points would be in this set too.

The line determined by these points should meet the boundary of $S $
which is in $\cup_{j\in N}A^j$. Thus the line would meet some set $A^i$ in a
point $a$. Suppose that $b_1$ is between $a$ and $b_2$. Let $c$ be the point
in $A^i$ at distance $\varepsilon_0$ from $b_2$. Consider the plane of dimension two
determined by the line $cb_2$ and the line $b_1b_2$. This plane meets the
supporting hyperplane to $A^i$ at $c$ and perpendicular on $cb_2$ in a line
$\lambda$
 which is perpendicular
to $cb_2$. Now, $a$ must be behind the supporting hyperplane, hence the line
$b_2b_1$ meets the line $\lambda$ in a point $d$ between $a$ and $b_2$. Thus
the triangle $dcb_2$ is rectangular at $c$. Since $B_{\varepsilon_0}$ is convex, we
can suppose without loss of generality that $b_1$ is on the segment $fb_2$,
where $f$ is the base of the perpendicular from $c$ to $b_1b_2$. But then the
distance from $b_1$ to $c$ is less then the distance of $b_2$ to $c$ which is
$\varepsilon_0$. This contradiction shows that $B_{\varepsilon_0}$ reduces to a point.

Let us observe now, that $\cup_{i\in N}A^i_{\varepsilon_0}=S$ since
for $v\in B_{\varepsilon_0}$  the simplexes $\co \{v,a_j:\;j\in N\setminus \{i\}\}
\subset A^i$ form a simplicial subdivision of $S$.

Take $u\in S,\;u\not= v.$ Then $u\notin B_{\varepsilon_0},$ and hence $u$
must be in some $A^k_{\varepsilon_0}$ and thus $d(u,A^k)\le \varepsilon_0$
and $u$ will be outside of some $A^j_{\varepsilon_0}$, hence $d(u,A^j)>\varepsilon_0.$
Therefore $u$ cannot be  equally spaced to every $A^i$, which shows that
$v$ is the single point with this property.

The same reasoning shows that the equality in the point 4. of the
lemma takes place.

Let be $A^1, A^2,...,A^{n+1}$ and $D^1, D^2,...,D^{n+1}$ be the
$\mathcal H$-families in the point 5. Then we have

$$A^i_{\varepsilon_0}\subset D^i_{\varepsilon_0},\;i \in N$$
with $\varepsilon_0$ defined at 1.
Hence
$$ \cap_{i\in N}A^i_{\varepsilon_0}\subset
\cap_{i\in N}D^i_{\varepsilon_0} \not= \emptyset,$$
and then
$$\varepsilon_1=\inf \{\varepsilon:\,\cap_{i\in N}D^i_\varepsilon \not=\emptyset\}\leq \varepsilon_0.$$

\end{proof}

\begin{lem}\label{n}
If the $n$-simplex $S $ is covered by $n+1$ closed convex sets
$C^1,\,C^2,...,C^{n+1}$ such that $S^i\subset
C^i$ for each $i\in N$, then
putting
$$C^i_t=(1-t)S^i+tC^i,\;t\in [0,1],\;i\in N,$$ and $\mathcal H_t = \{C^i_t:\;i\in N\}$,
we have the following assertions fulfilled:
\begin{enumerate}
\item There exists a $t_0\in (0,1]$ such that

(i) $\{C^i_t:\;i\in N\}$ is an $\mathcal H$-family for $t<t_0$;

(ii) $\{C^i_t:\;i\in N\}$ covers $S$ for $t\geq t_0$.
\item If $\varepsilon_t$ denotes the distance of the equally spaced point $v_t$ from
the members $C^i_t$ of the $\mathcal H$-family $\{C^i_t:\;i\in N\}$,
then $\varepsilon_t$ is
decreasing with $t$.

\item If $\{C^i_t:\;i\in N\}$ is an $\mathcal H$-family, then there exists
a neighborhood $W$ of $t$ such that $\{C^i_{t'}:\;i\in N\}$ is an
$\mathcal H$-family for any $t'\in W$.
\item $\delta_0=\inf \{\varepsilon_t:\,\mathcal H_t \;\textrm{is an}\;\mathcal H-\textrm{family}\}=0$
and there exists a sequence of $\varepsilon_t$-s converging to $0$.

\end{enumerate}

\end{lem}

\begin{proof}
Denoting with $\|.\|$ the Euclidean norm in $\mathbb R^n$ we have for $t\in [0,1]$ and
$x=(1-t)s+tc\in C^i_{t},$ ($s\in S^i,\;c\in C^i)$ that
$\|x-s\|=\|(1-t)s+tc-s\|=t\|c-s\|\leq td$ with $d$ the diameter of $S$.

Denote with $b$ the barycenter of $S$ and with $\delta$ the minimal distance
of $b$ from the faces $S^i$. Consider an arbitrary $x\in C^i_t$ represented in the above form.
Then
$$\|x-b\|\geq \|b-s\|-\|s-x\|\geq \delta -td,$$
whereby $\|x-b\|>0$ for $t>0$ sufficiently small. Varying $i$, we can get
a positive $t$ which is of this property for all $i\in N$.
Taking such a $t$, we have $b\notin C^i_t,\;i\in N$. Thus the later sets
form an $\mathcal H$-family.

We have $C^i_{t_1}\subset C^i_{t_2}$ for $t_1<t_2$.

Indeed, for an arbitrary
$(1-t_1)s+t_1c\in C^i_{t_1}$ we have
$$(1-t_1)s+t_1c =(1-t_2)s+t_2(\frac{t_2-t_1}{t_2}s+\frac{t_1}{t_2}c)\in C^i_{t_2}$$
since $S^i\subset C^i$.

Hence if $\{C^i_t:\;i\in N\}$ is an $\mathcal H$-family, then $\{C^i_{t'}:\;i\in N\}$ has
for $t'<t$ the same property.

From the considerations above and the fact that
$\{C^i_1:\;i\in N\}$ covers $S$ follows the existence of a $t_0\in (0,1]$
satisfying the requirements in point 1 of the lemma.

From the fact that we have $C^i_{t_1}\subset C^i_{t_2}$ for $t_1<t_2$
and the assertion 5 in Lemma \ref{kn}, it follows the assertion 2 of the lemma.

Suppose that $\{C^i_t:\;i\in N\}$ is an $\mathcal H$-family. Fix $i$ for
the moment and take $x=(1-t')s+t'c\in C^i_{t'}$ Let be $v_t$ the point in $S$
at the distance $\varepsilon_t$ from the sets $C^i_t$. Then
$$\|x-v_t\|= \|(1-t)s+tc+(t-t')(s-c)-v_t\|\geq \|(1-t)s+tc-v_t\|-|t-t'|\|s-c\|.$$
The obtained relation shows that for $|t-t'|$ sufficiently small the distance
of $v_t$ from all the members of $\{C^i_{t'}:\;i\in N\}$ is positive, which
concludes the proof of the point 3 of the lemma.

To prove the assertion 4 of the lemma, we assume the contrary: $\delta_0>0 $.

Consider the sequence $(t_m)$, $t_m<t_0,\;t_m\to t_0$. For every $m$ we determine
the equally spaced point $v_m=v_{t_m}$ to the sets $C^i_{t_m}$ and its distance
$\varepsilon_m=\varepsilon_{t_m}$ to these sets.
Passing to a subsequence if necessary, we can suppose that $v_m\to v\in S$.
We shall show, that
$$\|x-v\|\geq \delta_0\;\forall\;x\in C^i_{t_0},\;\forall i\in N.$$
Indeed, fixing $i$ for the moment and taking an arbitrary
$x=(1-t_0)s+t_0c\;(s\in S^i,\;c\in C^i)$,  we have
$$\|(1-t_m)s+t_mc-v\|\geq \|(1-t_m)s+t_mc-v_m\|-\|v_m-v\|\geq \varepsilon_m -\|v_m-v\|.$$
By passing to limit with $m$ we conclude
$$\|x-v\|=\|(1-t_0)s+t_0c-v\|\geq \delta_0.$$
Hence we have for $i\in N$ that
$$d(v,C^i_{t_0})\geq \delta_0.$$
The obtained relations show that $\{C^i_{t_0}:\;i\in N\}$
is an $\mathcal H$-family, which contradicts the definition of $t_0$. From the point 2 and the proof
above it follows also that the sequence $(\varepsilon_m)$ decreases to $0$ with
$m\to \infty$.
\end{proof}

{\it Proof Theorem \ref{main}}.

The proof of the first part of the theorem and the unicity of the
equally spaced point when the distance is positive follows from Lemma \ref{kn}.

Suppose that the family $\{A^i:\;\i\in N\}$ covers $S$. Then considering the sets
$$A^i_t=(1-t)S^i+tA^i,\;t\in [0,1]$$
we can determine according to Lemma \ref{n} the maximal number $t_0\in (0,1]$ such
that $A^i_t:\;i\in N$ form an $\mathcal H$-family for $t<t_0$ with
the equally spaced point $v_t$ from its members and the distance $\varepsilon_t$.
Then we can determine a sequence of distances $\varepsilon_{t_m}$ tending to $0$
with $m$. Consider the point $x^i_m\in A^i_{t_m}\subset A^i$ of the
distance $\varepsilon_{t_m}$ from $v_{t_m}$. Fix $j\in N$. Passing if necessary to a subsequence
we can suppose that $x^j_m\to x\in A^j$. Since $\varepsilon_{t_m}\to 0$ it follows that
$x^i_m\to x,\;\forall \;i\in N.$ Hence $x\in \cap_{i\in N}A^i$ and thus
$d(x,A^i)=0,\;i\in N.$ This completes the proof of the theorem.

{\it Proof of Corollary \ref{dist}}

Since every $S^i$ is in some $B^\alpha$ there must exist
a subfamily
$$\{B^{\alpha_1}, B^{\alpha_2},...,B^{\alpha_{n+1}}\}$$
so as to have $S^i\subset B^{\alpha_i}$. Obviously
$\{B^{\alpha_1}, B^{\alpha_2},...,B^{\alpha_{n+1}}\}$
cannot cover $S$, and hence according the firs part of Theorem
\ref{main}, there exists a unique $v\in S$ which is at the distance
$\varepsilon >0$
to every $B^{\alpha_i}$. Taking now an arbitrary other $B^\alpha$,
according the condition of the corollary, $v$ must be at the same
distance $\varepsilon $ to it.

{\it Proof of Corollary \ref{hell}.}

Let $\mathcal B= \{B^\alpha:\;\alpha\in I\}$ be a face covering family
of closed convex sets in $S$. If the cardinality of $I$ is $\leq n,$
then we have nothing to prove since any $k\leq n$ maximal faces
of $S$ have a common point which will be a common point for $\mathcal B$.
Hence we can suppose that $\mathcal B$ contain at least
$n+1$ members.

Denote then in the following with $N$ the set $N = \{1, 2, ..., n+1\}$.
Let now $\mathcal B= \{B^\alpha:\;\alpha\in I\}$ be a face covering family
of closed convex sets in $S$ with cardinality of $I \geq n+1$.
Let us suppose that a subfamily 
$\mathcal B_0=\{B^{\alpha_1},B^{\alpha_2},...,B^{\alpha_{n+1}}\}$
 has a nonempty intersection and that $\mathcal B_0$ covers the boundary of $S$. Consider then 
a point $x \in \bigcap_{i \in N}B^{\alpha_i}$ and $v$ an arbitrary point in the simplex $S$.
The halfline with the origin in the point $x$ and going through the point $v$ must then
intersect the boundary of the simplex $S$ in a point $y$. Because the subfamily $\mathcal B_0$
is covering the boundary of $S$ there is an $i \in N$ such that $y \in B^{\alpha_i}$. We have also
$x \in B^{\alpha_i}$. From the convexity of the set $B^{\alpha_i}$ follows then $v \in B^{\alpha_i}$.
Therefore the simplex $S$ is also covered by the subfamily $\mathcal B_0$. 

Consider an arbitrary subfamily $\mathcal B_0= \{B^{\alpha_1},B^{\alpha_2},...,B^{\alpha_{n+1}}\}$ of
$\mathcal B$. If some $S^i$ is not covered by it, then the vertex $a_i$ which is
element of every $S^j,\;j\not=i$ must be element of each member of the family.

Suppose that the subfamily covers the boundary of $S$. If some $B^{\alpha_i}$ contains
two different maximal faces of $S$, then by convexity it covers $S$ and then
it contains the vertex $a_k$ contained in the intersection of the sets $B^{\alpha_j},\;j\not=i.$

If no $B^{\alpha_i}$ contain two different maximal faces, then we can
consider that $S^i\subset B^{\alpha_i},\;i\in N$, and as soon by hypothesis
the family $\mathcal B_0$ covers $S$, we have according Theorem \ref{main} that
it has a nonempty intersection.

In conclusion, each family of $n+1$ members of $\mathcal B$ has a nonempty
intersection. Hence from the theorem of Helly, the whole $\mathcal B$ has
nonempty intersection.

\end{document}